\documentclass[12pt]{article}
\usepackage{mathrsfs}
\usepackage{bbm}
\usepackage{fancyhdr,graphicx}
\usepackage{amsfonts}
\usepackage{amssymb}
\usepackage{amsthm}
\usepackage{amsmath}
\usepackage{newlfont}
\usepackage[toc,page,title,titletoc,header]{appendix}

\newtheorem{thm}{Theorem}[section]
\newtheorem{Lemma}[thm]{Lemma}
\newtheorem{cor}[thm]{Corollary}

\makeatletter
\long\def\@makecaption#1#2{%
 \vskip\abovecaptionskip
  \sbox\@tempboxa{{#1.}\quad #2}%
   \ifdim \wd\@tempboxa >\hsize
    { #1.}\quad #2\par
     \else
  \global \@minipagefalse
   \hb@xt@\hsize{\hfil\box\@tempboxa\hfil}%
   \fi
   \vskip\belowcaptionskip}
\makeatother

\usepackage{latexsym, bm}
\title{On the restricted matching of graphs in surfaces\footnote{This
 work was supported by NSFC (grant no. 10831001).}}

\author{Qiuli Li, Heping Zhang\footnote{The
Corresponding author.} \\
\small{School of Mathematics and Statistics, Lanzhou University,
Lanzhou, Gansu 730000, P. R. China}\\
\small{E-mail addresses:  liql06@lzu.cn, zhanghp@lzu.edu.cn}}

\date{}

\begin{document}
\maketitle
\begin{abstract} A connected graph $G$ with at
least $2m+2n+2$ vertices is said to have property $E(m,n)$ if, for
any two disjoint matchings $M$ and $N$ of size $m$ and $n$
respectively, $G$ has a perfect matching $F$ such that $M\subseteq
F$ and $N\cap F=\varnothing$. In particular, a graph with $E(m,0)$
is $m$-extendable. Let $\mu(\Sigma)$ be the smallest integer $k$
such that no graphs embedded on a surface $\Sigma$ are
$k$-extendable. Aldred and Plummer have proved that no graphs
embedded on the surfaces $\Sigma$ such as the sphere, the projective
plane, the torus, and the Klein bottle are
 $E(\mu(\Sigma)-1,1)$.  In this paper, we show that
this result always holds for any surface. Furthermore, we obtain
that if a graph $G$ embedded on a surface has sufficiently many
vertices, then $G$ has no property $E(k-1,1)$ for each integer
$k\geq 4$, which implies that $G$ is not $k$-extendable. In the case
of $k=4$, we get immediately a main result that  Aldred et al.
recently obtained. \vspace{0.3cm}


\end{abstract}

\section{Introduction}
A {\em matching} of a graph $G$ is a set of independent edges of $G$
and a matching is called {\em perfect} if it covers all vertices of
$G$. A connected graph $G$ with at least $2m+2n+2$ vertices is said
to have property $E(m,n)$ (or abbreviated as $G$ is $E(m,n)$) if,
for any two disjoint matchings $M$, $N \subseteq E(G)$ of size $m$
and $n$ respectively, there is a perfect matching $F$ such that
$M\subseteq F$ and $N\cap F=\varnothing$. It is obvious that a graph
with $E(0,0)$ has a perfect matching. Since
  properties $E(m,0)$ and $m$-extendability are
equivalent,  property $E(m,n)$ is somewhat a generalization of
$m$-extendability. The concept of $m$-extendable graphs was
gradually evolved from the study of elementary bipartite graphs and
matching-covered graphs (i.e. each edge belongs to a perfect
matching ) and introduced by M.D. Plummer \cite{Plummer80} in 1980.
For extensive studies on $m$-extendable graphs,  see two surveys
\cite{Plummer94} and \cite{Plummer96}. A basic property is stated as
follows.

\begin{Lemma}\label{n+1connected}(\cite{Plummer80})
Every $m$-extendable graph is $(m+1)$-connected.
\end{Lemma}

For a vertex $v$ of a graph $G$, let $N(v)$ denote the neighborhood
of $v$, i.e., the set of vertices adjacent to $v$ in $G$, and
$G[N(v)]$ the subgraph of $G$ induced by $N(v)$.
\begin{Lemma}\label{G[N(v)]}(\cite{Dean92})
 Let $v$ be a vertex of degree $m+t$ in an $m$-extendable graph
$G$. Then $G[N(v)]$ does not contain a matching of size $t$.
\end{Lemma}

Porteous and Aldred \cite{M. Porteous96} introduced the concept of
property  $E(m,n)$ and  focussed on when the implication
$E(m,n)\rightarrow E(p,q)$ does and does not hold. From then on, the
possible implications among the properties $E(m,n)$ for various
values of $m$ and $n$ are studied in \cite{A.
McGregor-Macdonald00,M. Porteous95,M. Porteous96}. The following
three non-trivial results  will be used later.

\begin{Lemma}\label{E(m,n)--E(m,0)}(\cite{M. Porteous96})
If a graph $G$ is $E(m,n)$, then it is $E(m,0)$.
\end{Lemma}

\begin{Lemma}\label{E(m,n)--E(m-1,n)}(\cite{M. Porteous96})
If a graph $G$ is $E(m,n)$, then it is $E(m-1,n)$.
\end{Lemma}

\begin{Lemma}\label{E(m,0)--E(m-1,1)}(\cite{M. Porteous96})
If a graph $G$ is $E(m,0)$ for $m\geq1$, then it is $E(m-1,1)$.
\end{Lemma}

The converse of  Lemma \ref{E(m,0)--E(m-1,1)} does not hold. For
example, the join graph $\overline{K_{2}}+K_{2m}$, obtained by
joining each of two vertices  to each vertex of the complete graph
$K_{2m}$ with edges, has property $E(m-1,1)$, but is not
$m$-extendable.

 A {\em surface} is a connected compact Hausdorff space which is
locally homeomorphic to an open disc in the plane. If a surface
$\Sigma$ is obtained from the sphere by adding some number $g\geq0$
of handles (resp. some number $\bar{g}>0$ of crosscaps), then it is
said to be {\em orientable} of genus $g=g(\Sigma)$ (resp. {\em
non-orientable} of genus $\bar{g}=\bar{g}(\Sigma)$). We shall follow
the usual notation of the surface of orientable genus $g$ (resp.
non-orientable genus $\bar{g}$) by $S_{g}$ (resp. $N_{\bar{g}}$).

Let $\mu(\Sigma)$ be the smallest integer $k$ such that no graphs
embedded on the surface $\Sigma$ are $k$-extendable. Dean
\cite{Dean92} presented an elegant formula  that
\begin{equation}\label{Dean}\mu(\Sigma)=2+\lfloor \sqrt{4-2\chi(\Sigma)}\rfloor,\end{equation}
where $\chi(\Sigma)$ is the Euler characteristic of a surface
$\Sigma$, i.e. $\chi(\Sigma)=2-2g$ if $\Sigma$ is an orientable
surface of genus $g$ and $\chi(\Sigma)=2-\bar{g}$ if $\Sigma$ is a
non-orientable surface of genus $\bar g$. For the surfaces $\Sigma$
with small genus such as the sphere, the projective plane, the torus
and the Klein bottle, the following results show that no graphs
embedded on $\Sigma$ are $E(\mu(\Sigma)-1,1)$.

\begin{Lemma}\label{1,2,3,4}
(i)(\cite{Aldred01}) No planar graph is $E(2,1)$;\\
(ii) (\cite{R.E.L. Aldred08})No projective planar graph is $E(2,1)$;\\
(iii) (\cite{R.E.L. Aldred08})If $G$ is  toroidal, then $G$ is not $E(3,1)$;\\
(iv) (\cite{R.E.L. Aldred08})If $G$ is embedded on the Klein bottle,
then $G$ is not $E(3,1)$.
\end{Lemma}

In this paper we obtain the following general result, which will be
proved in next section.

\begin{thm}\label{general}
For any surface $\Sigma$, no graphs embedded on $\Sigma$ are
$E(\mu(\Sigma)-1,1)$.
\end{thm}

Furthermore, we obtain that if a graph $G$ embedded on a surface has
enough many vertices, then $G$ has no property $E(k-1,1)$ for each
integer $k\geq 4$. Precisely, we have the following result; its
proof  will be given in Section 3.

\begin{thm}\label{E(k-1,1)}
Let $G$ be a graph with genus $g$ (resp. non-orientable genus
$\bar{g}$). Then if $|V(G)|\geq \lfloor \frac{8g-8}{k-3}\rfloor+1$
(resp. $|V(G)|\geq \lfloor \frac{4\bar{g}-8}{k-3}\rfloor+1$), $G$ is
not $E(k-1,1)$ for each integer $k\geq4$.
\end{thm}

Combining Theorem \ref{E(k-1,1)} with  Lemma \ref{E(m,0)--E(m-1,1)},
we have an immediate consequence as follows.

\begin{cor}(\cite{Zhang08})
Let $G$ be any connected graph of genus $g$ (resp. non-orientable
genus \={g}). Then if $|V(G)|\geq \lfloor\frac{8g-8}{k-3}\rfloor+1$
(resp. $\lfloor\frac{4\bar{g}-8}{k-3}\rfloor+1$) for any integer
$k\geq4$, $G$ is not $k$-extendable.
\end{cor}
In particular, if we put $k=4$ in the corollary, we can obtain the
following result which is also a main theorem that Aldred et al.
recently obtained.

\begin{cor}(\cite{Aldred08})
Let $G$ be any connected graph of genus $g$ (resp. non-orientable
genus \={g}). Then if $|V(G)|\geq 8g-7$ (resp. $4\bar{g}-7$), $G$ is
not 4-extendable.
\end{cor}

\section{Proof of Theorem \ref{general} }

 For a graph $G$, the {\em genus} $\gamma(G)$ (resp. {\em non-orientable
genus} $\bar{\gamma}(G)$) of it is the minimum genus (resp.
non-orientable genus) of all orientable (resp. non-orientable)
surfaces in which $G$ can be embedded. An embedding $\tilde{G}$ of a
graph $G$ on an orientable surface $S_{k}$ (resp. a non-orientable
surface $N_{k}$) is said to be {\em minimal} if $\gamma(G)=k$ (resp.
$\bar{\gamma}(G)=k$) and {\em 2-cell} if each component of
$\Sigma-\tilde{G}$ is homeomorphic to an open disc.

\begin{Lemma}(\cite{J.W.T.63})\label{J.W.T.630}
Every minimal orientable
 embedding of a graph $G$ is a $2$-cell embedding.
\end{Lemma}

\begin{Lemma}(\cite{Parson87})\label{Parson870}
Every graph $G$ has a
 minimal non-orientable embedding which is 2-cell.
\end{Lemma}

Let $v$ be any vertex of a graph $G$ embedded on an orientable
surface of genus $g$ (resp. a non-orientable surface of genus
$\bar{g}$). Define the {\em Euler contribution} of the vertex $v$ to
be
\begin{align}\label{control}
\phi(v)=1-\frac{{\deg}(v)}{2}+\sum_{i=1}^{{\deg}(v)}\frac{1}{f_{i}},
\end{align}
\noindent where the sum runs over the face angles at vertex $v$,
$f_{i}$ denotes the size of the $ith$ face at $v$ and  ${\deg}(v)$
denotes the degree of $v$.

\begin{Lemma}(\cite{Lebesgue40})\label{2-cell}
Let $G$ be a connected graph  2-celluarlly embedded on some surface
$\Sigma$ of orientable genus $g$ (resp. non-orientable genus
$\bar{g}$). Then $\sum_{v}{\phi(v)}=\chi(\Sigma)$.
\end{Lemma}

For a vertex $v$, it is called a {\em control point} if $\phi(v)\geq
\frac{\chi(\Sigma)}{|V(G)|}$. If $G$ is 2-cellularly embedded on the
surface $\Sigma$, then $G$ must have at least one control point by
Lemma \ref{2-cell}.




Let $\delta(G)$ denote the minimum degree of the vertices in $G$.
The following lemma is a simple observation, which gives a lower
bound of $\delta(G)$ of a graph $G$ with $E(m,1)$.
\begin{Lemma}\label{minimum degree}
If a graph $G$ is $E(m,1)$ for $m\geq 1$, then $\delta(G)\geq m+2$.
\end{Lemma}
\begin{proof}
By Lemma \ref{E(m,n)--E(m,0)}, $G$ is $E(m,0)$. Moreover,
$\delta(G)\geq m+1$ by Lemma \ref{n+1connected}. Suppose to the
contrary that there exists a vertex $v$ with degree $m+1$. Then
$G[N(v)]$ cannot contain a matching of size 1 by Lemma
\ref{G[N(v)]}; that is, $N(v)$ is an independent set of $G$. Let
$N(v)=\{v_{1},v_{2},...,v_{m+1}\}$, $V=\{v_{1},v_{2},...,v_{m}\}$
and $R=V(G)\setminus N[v]$, where $N[v]=N(v)\cup \{v\}$. Let
$G[V,R]$ be the induced bipartite graph of $G$ with bipartition $V$
and $R$. Then every vertex in $V$ is adjacent to at least $m$
vertices in $R$. Hence it can easily be seen that $G[V,R]$ has a
matching $M$ of size $m$ saturating $V$. Let $N=\{vv_{m+1}\}$.
Obviously, there is no perfect matching $F$ of $G$ satisfying that
$M\subseteq F$ and $N \cap F=\varnothing$. This contradicts that $G$
is $E(m,1)$.
\end{proof}

\noindent{\em Proof of Theorem \ref{general} }.   Since
$\mu(\Sigma)$ increases  as $g$ (resp. $\bar{g}$) does and a graph
embedded on a surface with small genus must be embedded on some
surface with larger genus, it suffices to prove that any graph
minimally embedded on the surface $\Sigma$ is not
$E(\mu(\Sigma)-1,1)$ by Lemma \ref{E(m,n)--E(m-1,n)}. In the
following, we may assume that $G$ is minimally and 2-cell embedded
on the surface $\Sigma$ by Lemmas \ref{J.W.T.630} and
\ref{Parson870}.

By Lemma \ref{1,2,3,4}, the theorem holds for the surfaces $S_{0}$,
$S_{1}$, $N_{1}$ and $N_{2}$. Hereafter, we will restrict our
considerations on the other surfaces $\Sigma$. Consequently,
$\chi(\Sigma)\leq-1$ and $\mu(\Sigma)\geq4$.

Suppose to the contrary that $G$ is $E(\mu(\Sigma)-1,1)$. Then
$|V(G)|\geq 2(\mu(\Sigma)+1)$, and $\delta(G)\geq
\mu(\Sigma)+1\geq5$ by Lemma \ref{minimum degree}. Since $G$ is a
2-cell embedding on the surface $\Sigma$, it has a control point
$v$. Let $y:={\deg}(v)$ and let $x$ be the number of
 the triangular faces at $v$.

{\bf Claim 1.} $G$ is not $E(y-\lceil\frac{x}{2}\rceil,1)$.

If $x=y$ and $y$ is odd,  then there is a matching of size
 $\lfloor\frac{x}{2}\rfloor$ in $G[N(v)]$. Hence $G$ is not $E(\lfloor\frac{x}{2}\rfloor,1)$,
 that is, $G$ is not $E(y-\lceil\frac{x}{2}\rceil,1)$. Otherwise, there is a
matching of size $\lceil\frac{x}{2}\rceil$ in $G[N(v)]$. Then $G$ is
not $(y-\lceil\frac{x}{2}\rceil)$-extendable by
 Lemma \ref{G[N(v)]}. Hence $G$ is not $E(y-\lceil\frac{x}{2}\rceil,1)$ by Lemma
 \ref{E(m,n)--E(m,0)}. So the claim always holds.

By Eq. (\ref{control}) and $\phi(v)\geq
\frac{\chi(\Sigma)}{|V(G)|}$, we have

$$\frac{y}{2}\leq
1+\sum_{i=1}^{y}\frac{1}{f_{i}}-\frac{\chi(\Sigma)}{|V(G)|}\leq1+\frac{x}{3}+\frac{y-x}{4}-\frac{\chi(\Sigma)}{2(\mu(\Sigma)+1)},$$
which implies that
$$y\leq\frac{x}{3}+4-\frac{2\chi(\Sigma)}{\mu(\Sigma)+1}.$$
Let
\begin{equation}\label{c}
c:=4-\frac{2\chi(\Sigma)}{\mu(\Sigma)+1}.\end{equation} Then $c>4$
and $y-\lceil\frac{x}{2}\rceil\leq y-\frac{x}{2}\leq
y-\frac{x}{3}\leq c$.

{\bf Claim 2.} $G$ is not $E(\lfloor c\rfloor-1,1)$.

If $y-\lceil\frac{x}{2}\rceil\leq y-\frac{x}{2}\leq c-1$, then $G$
is not $E(\lfloor c\rfloor-1,1)$ by Lemma \ref{E(m,n)--E(m-1,n)} and
Claim 1.

In what follows we suppose that $y-\frac{x}{2}> c-1$. Combining this
with $y-\frac{x}{3}\leq c$, we have that $x\leq 5$, and all possible
cases of pairs of non-negative integers $(x,y)$ are as follows:

$(0,\lfloor c\rfloor), (1,\lfloor c\rfloor), (1,\lfloor c\rfloor+1),
(2,\lfloor c\rfloor+1), (3,\lfloor c\rfloor+1), (4,\lfloor
c\rfloor+2), \mbox{ and } (5,\lfloor c\rfloor+2)$.

Suppose to the contrary that $G$ is $E(\lfloor c\rfloor-1,1)$. Then
$G$ is $(\lfloor c\rfloor-1)$-extendable by Lemma
\ref{E(m,n)--E(m,0)} and $\delta(G)\geq \lfloor c\rfloor+1$ by Lemma
 \ref{minimum degree}. Hence the first two
cases $(0,\lfloor c\rfloor)$ and $(1,\lfloor c\rfloor)$ are
impossible. If $y=\lfloor c\rfloor+1$, since $\deg (v)=y=(\lfloor
c\rfloor-1)+2$, $G[N(v)]$ cannot contain a matching of size $2$ by
Lemma \ref{G[N(v)]}.

For convenience, let  $v_{1},v_{2},...,v_{y}$ be the vertices
adjacent to $v$ arranged clockwise at $v$ in $G$. Similar to the
notation in the proof of Lemma \ref{minimum degree}, let
$R=V(G)\setminus N[v]$ and $G[V,R]$ denote the induced bipartite
graph of $G$ with bipartition $V$ and $R$, where $V\subseteq
V(G)\setminus R$.

 For $(x,y)=(1,\lfloor c\rfloor+1)$, $G[N(v)]$ cannot contain a matching of size
 $2$. Assume that the triangular face at $v$ is $vv_{1}v_{2}$.  Hence  each $v_{i}$,
 $3\leq i \leq \lfloor c\rfloor+1$,  can only be adjacent to
$v_{1}$ and $v_{2}$ in $N(v)$, and has at least $\lfloor c\rfloor-2$
adjacent vertices in $R$. Let $V:=\{v_{3},v_{4},...,v_{\lfloor
c\rfloor}\}$. Then there is a matching $M'$ of size $\lfloor
c\rfloor-2$ in  $G[V,R]$. Let $M:=M'\cup \{v_{1}v_{2}\}$ and
$N:=\{vv_{\lfloor c\rfloor+1}\}$. Obviously, there is no perfect
matching $F$ satisfying that $M\subseteq F$ and $N\cap
F=\varnothing$, a contradiction.

For $(x,y)=(2,\lfloor c\rfloor+1)$,  $G[N(v)]$ cannot contain a
matching of size $2$, and the two triangular faces at $v$ must be
adjacent. Hence we can assume that they are $vv_{1}v_{2}$ and
$vv_{2}v_{3}$. Each $v_{i}$,  $4\leq i \leq \lfloor c\rfloor+1$, can
only be adjacent to $v_{2}$ in $N(v)$, and has at least $\lfloor
c\rfloor-1$ adjacent vertices in $R$. Let
$V:=\{v_{4},v_{5},...,v_{\lfloor c\rfloor+1}\}$. Then we can find a
matching $M'$ of size $\lfloor c\rfloor-2$ in $G[V,R]$. Let
$M=M'\cup \{v_{1}v_{2}\}$ and $N=\{vv_{3}\}$. Consequently, it is
impossible to find a perfect matching $F$ satisfying that
$M\subseteq F$ and $N\cap F=\varnothing$, a contradiction.

For $(x,y)=(3,\lfloor c\rfloor+1)$, $G[N(v)]$ contains a matching of
size $\lceil\frac{3}{2}\rceil$=2. This would be impossible.

If $y=\lfloor c\rfloor+2$, since $\deg(v)=(\lfloor c\rfloor-1)+3$,
$G[N(v)]$ cannot contain a matching of size $3$ by Lemma
\ref{G[N(v)]}. Hence $(x,y)=(5,\lfloor c\rfloor+2)$ would also be
impossible since $G[N(v)]$ contains a matching of size
$\lceil\frac{5}{2}\rceil$=3.

For the remaining case $(x,y)=(4,\lfloor c\rfloor+2)$,  $G[N(v)]$
cannot contain a matching of size $3$. Then the four triangular
faces at $v$ must have the following two cases. {\em Case 1}. The
four triangular faces are $vv_{i}v_{i+1}$, $1\leq i \leq 4$. Each
$v_{i}$, $6 \leq i \leq \lfloor c\rfloor+2$,  can only be adjacent
to $v_{2}$ or $v_{4}$, and has at least $\lfloor c\rfloor-2$
adjacent vertices in $R$. Let  $V:=\{v_{6},v_{7},...,v_{\lfloor
c\rfloor+2}\}$. Then we can find a matching $M'$ of size $\lfloor
c\rfloor-3$ in $G[V,R]$. Let $M=M'\cup \{v_{1}v_{2}, v_{4}v_{5}\}$
and $N=\{vv_{3}\}$. Obviously, there is no perfect matching $F$
satisfying that $M\subseteq F$ and $N\cap F=\varnothing$, a
contradiction. {\em Case 2}. The four triangular faces are
$vv_{i}v_{i+1}$ for  $i=1, 2$ and $vv_{j}v_{j+1}$ for  $j=t, t+1$,
where $t\neq1,2,3, \lfloor c\rfloor+1, \lfloor c\rfloor+2$. Then
each $v_{i}$, $i\neq1,2,3,t,t+1$ and $t+2$,  can only be adjacent to
$v_{2}$ and $v_{t+1}$, and  has at least $\lfloor c\rfloor-2$
adjacent vertices in $R$.  $v_{3}$ can only be adjacent to
$v_{1},v_{2}$ and $v_{t+1}$ in $G[N(v)]$, and  has at least $\lfloor
c\rfloor-3$ adjacent vertices in $R$. Let
$V:=N(v)-\{v_{1},v_{2},v_{t},v_{t+1},v_{t+2}\}$. Then we can find a
matching $M'$ of size $\lfloor c\rfloor-3$ in $G[V,R]$. Set
$M=M'\cup \{v_{1}v_{2}, v_{t}v_{t+1}\}$ and $N=\{vv_{t+2}\}$.
Obviously, there is no perfect matching $F$ satisfying that
$M\subseteq F$ and $N\cap F=\varnothing$, a contradiction. Hence the
claim holds.

{\bf Claim 3.} $\lfloor c\rfloor\leq\mu(\Sigma)$.

In fact, the inequality was stated in \cite{Dean92} without proof.
Here we present a simple proof.

Owing to the expressions (\ref{Dean}) and (\ref{c}) of $\mu(\Sigma)$
and $c$, it suffices to prove that $$\lfloor
4-\frac{2\chi(\Sigma)}{3+\lfloor
\sqrt{4-2\chi(\Sigma)}\rfloor}\rfloor\leq2+\lfloor
\sqrt{4-2\chi(\Sigma)}\rfloor.$$ Then we have the following
implications:

\hspace*{0.9cm}$\lfloor 4-\frac{2\chi(\Sigma)}{3+\lfloor
\sqrt{4-2\chi(\Sigma)}\rfloor}\rfloor\leq2+\lfloor
\sqrt{4-2\chi(\Sigma)}\rfloor$

$\Longleftrightarrow 3-\frac{2\chi(\Sigma)}{3+\lfloor
\sqrt{4-2\chi(\Sigma)}\rfloor}< 2+\lfloor
\sqrt{4-2\chi(\Sigma)}\rfloor$

$\Longleftrightarrow 9+3\lfloor
\sqrt{4-2\chi(\Sigma)}\rfloor-2\chi(\Sigma)<(\lfloor
\sqrt{4-2\chi(\Sigma)}\rfloor)^{2}+5\lfloor
\sqrt{4-2\chi(\Sigma)}\rfloor+6$

$\Longleftrightarrow
4-2\chi(\Sigma)<(\lfloor\sqrt{4-2\chi(\Sigma)}\rfloor+1)^{2}$\\
Since the last inequality obviously holds,  the claim follows.

By the above arguments, we know that  $G$ is $E(\mu(\Sigma)-1,1)$
but not $E(\lfloor c\rfloor-1,1)$. Hence $\lfloor
c\rfloor-1>\mu(\Sigma)-1$ by Lemma \ref{E(m,n)--E(m-1,n)}, which
contradicts   Claim 3.  \hfill $\square$

\section{Proof of Theorem \ref{E(k-1,1)}}

Suppose to the contrary that $G$ is $E(k-1,1)$. Then $G$ is
$E(k-1,0)$ and $\delta(G)\geq k+1$ by Lemmas \ref{E(m,n)--E(m,0)}
and \ref{minimum degree}. We can assume that $G$ is a 2-cell
embedding on the surface $S_{g}$ (resp. $N_{\bar{g}}$) by Lemmas
\ref{J.W.T.630} and \ref{Parson870}. In the following, we mainly
prove that $\phi(v)\leq -\frac{k-3}{4}$ for any vertex $v\in V(G)$.
If it holds, by Lemma \ref{2-cell} we have
$\chi(\Sigma)=\sum_v\phi(v)\leq -\frac{k-3}{4}|V(G)|$, which implies
that $|V(G)|\leq \frac{-4\chi(\Sigma)}{k-3}$ for $k\geq 4$. This
contradiction to the condition establishes the theorem.

Let $d={\deg}(v)=k+m$. Then $m\geq1$. For convenience, we assume
that $v_{1},v_{2},...,v_{d}$ are the vertices adjacent to $v$
arranged clockwise at $v$ in $G$. There are three cases to be
considered.



{\bf{Case 1.}} $m\geq3$. Since  $d=(k-1)+m+1$ and  $G$ is
$(k-1)$-extendable,  $G[N(v)]$ cannot contain a matching of size
$m+1$ by Lemma \ref{G[N(v)]}. If there are at most $2m$ triangular
faces at $v$, then we have
$$\phi(v)\leq1-\frac{d}{2}+\frac{2m}{3}+\frac{k+m-2m}{4}=
\frac{-3k-m+12}{12}\leq \frac{-3k-3+12}{12}=\frac{3-k}{4}.$$

Otherwise,  there are exactly $2m+1$ triangular faces at $v$ and
$d=2m+1=2k-1$. Let $M:=\{v_{i}v_{i+1}| 1\leq i \leq 2m-1, i \text{
is odd}\}$ and $N:=\{vv_{2m+1}\}$. Then there exists no perfect
matching $F$ such that $M\subseteq F$ and $N\cap F=\varnothing$. But
$G$ is $E(k-1,1)$, a contradiction.

{\bf{Case 2.}} $m=2$. Since  $d=(k-1)+3$, $G[N(v)]$ cannot contain a
matching of size 3 by Lemma \ref{G[N(v)]}. Hence there are at most
four triangular face at $v$. It suffices to prove that there are at
most three triangular face at $v$. If so, we have that $\phi(v)\leq
1-\frac{k+2}{2}+\frac{3}{3}+\frac{k-1}{4}=\frac{3-k}{4}$.

Suppose to the contrary that there are exactly four triangular faces
at $v$. Then there are two cases to be considered. \\ {\em Subcase
2.1}. The four triangular faces are $vv_{i}v_{i+1}$, where $1\leq i
\leq 4$. Then  each $v_{i}$, $6 \leq i \leq k+2$,  can only be
adjacent to $v_{2}$ and $v_{4}$ and has at least $k+1-3=k-2$
adjacent vertices in $R$, where $R=V(G)\setminus N[v]$. Let
$V:=\{v_{6},v_{7},...,v_{k+2}\}$. Then we can find a matching $M'$
of size $k-3$ in the induced bipartite graph $G[V,R]$ of $G$. Let
$M:=M'\cup \{v_{1}v_{2}, v_{4}v_{5}\}$ and $N:=\{vv_{3}\}$.
Obviously, there is no perfect matching $F$ satisfying that
$M\subseteq F$ and $N\cap F=\varnothing$, which
contradicts the supposition that  $G$ has property $E(k-1,1)$. \\
{\em Subcase 2.2}. The four triangular faces are $vv_{i}v_{i+1}$,
$1\leq i \leq 2$, and $vv_{j}v_{j+1}$, $t\leq j \leq t+1$, where
$t\neq1,2,3,k+1,k+2$. Then for each $v_{i}$, $i\neq1,2,3,t,t+1$ and
$t+2$, it can only be adjacent to $v_{2}$ and $v_{t+1}$, and  has at
least $k-2$ adjacent vertices in $R$. For the vertex $v_{3}$, it can
only be adjacent to $v_{1},v_{2}$ and $v_{t+1}$ in $N(v)$. Then it
has at least $k-3$ adjacent vertices in $R$. Let
$V=:N(v)-\{v_{1},v_{2},v_{t},v_{t+1},v_{t+2}\}$. Then we can find a
matching $M'$ of size $k-3$ in the induced bipartite graph $G[V,R]$
of $G$. Let $M:=M'\cup \{v_{1}v_{2}, v_{t}v_{t+1}\}$ and
$N:=\{vv_{t+2}\}$. Consequently, it would be impossible to find a
perfect matching $F$ satisfying that $M\subseteq F$ and $N\cap
F=\varnothing$, a contradiction.

{\bf Case 3.} $m=1$. Since $d=(k-1)+2$,  $G[N(v)]$ cannot contain a
matching of size 2. Then there are at most two  triangular faces at
$v$. It suffices to prove that  there are no triangular faces at
$v$. If so,
$\phi(v)=1-\frac{d}{2}+\sum_{i=1}^{d}\frac{1}{f_{i}}\leq1-\frac{d}{2}+\frac{d}{4}=\frac{3-k}{4}$.

If there is exactly one triangular face at $v$,  suppose that it is
$vv_{1}v_{2}$. Since  $G[N(v)]$ cannot contain a matching of size 2,
each vertex $v_{i}$, where $3 \leq i \leq k$,  can only be adjacent
to $v_{1}$ and $v_{2}$ in $N(v)$. Consequently, it is adjacent to at
least $k-2$ vertices in $R$. Let $V=\{v_{3},v_{4},...,v_{k}\}$. Then
we can find a matching $M'$ of size $k-2$ in the induced bipartite
graph $G[V,R]$ of $G$. Set $M=M'\cup \{v_{1}v_{2}\}$ and
$N=\{vv_{k+1}\}$. Then there is no perfect matching $F$ satisfying
that $M\subseteq F$ and $N\cap F=\varnothing$, a contradiction.

Otherwise,  there are exactly two triangular faces at $v$, which
must be adjacent faces, say $vv_{i}v_{i+1}$, where $i=1,2$. Then
each  vertex $v_{i}$, where $4 \leq i \leq k+1$,  can only be
adjacent to $v_{2}$ in $N(v)$, and  is adjacent to at least $k-1$
vertices in $R$. Let $V:=\{v_{4},v_{5},...,v_{k+1}\}$. Then we can
find a matching $M'$ of size $k-2$ in the induced bipartite graph
$G[V,R]$ of $G$. Set $M:=M'\cup \{v_{1}v_{2}\}$ and $N:=\{vv_{3}\}$.
Then there is no perfect matching $F$ satisfying that $M\subseteq F$
and $N\cap F=\varnothing$, a contradiction.



\begin{thebibliography}{20}
\parskip -0.1cm

\bibitem{Aldred08} R.E.L. Aldred, Ken-ichi Kawarabayashi, M.D. Plummer, On the matching extendability of graphs in surfaces,
 J. Combin. Theory Ser. B 98 (2008) 105-115.
\bibitem{Aldred01} R.E.L. Aldred and M.D. Plummer, On restricted
matching extension in planar graphs, in: 17th British Combin.
Conference (Canterbury 1999), Discrete Math. 231 (2001) 73-79.
\bibitem{R.E.L. Aldred08} R.E.L. Aldred and M.D. Plummer, Restricted matching in graphs of small genus,
 Discrete Math. 308 (2008) 5907-5921.
\bibitem{Dean92} N. Dean, The matching extendability of surfaces, J. Combin. Theory Ser. B 54 (1992) 133-141.
\bibitem{Lebesgue40} H. Lebesgue, Quelques cons\'{e}quences simples de
la formule d'Euler, J. Math. 9 (1940) 27-43.
\bibitem{A. McGregor-Macdonald00} A. McGregor-Macdonald, The
$E(m,n)$ property, M.S. Thesis, University of Otago, 2000.
\bibitem{Parson87} T. Parsons, G. Pica, T. Pisanski and A. Ventre, Orientably
simple graphs, Math. Slovaca 37 (1987) 391-394.
\bibitem{Plummer89} M.D. Plummer, A theorem on matchings in the plane, in:
L.D. Andersen, et al. (Eds.), Proceedings of a Conference in Memory
of Gabriel Dirac, Ann. Diacrete Math. Vol. 41, North-Holland,
Amsterdam, 1989, pp. 347-354.
\bibitem{Plummer94} M.D. Plummer, Extending matchings in graphs: a
survey, Discrete Math. 127 (1994) 277-292.
\bibitem{Plummer96} M.D. Plummer, Extending matchings in graphs: an
update, Cong. Numer. 116 (1996) 3-32.
\bibitem{Plummer86genus} M.D. Plummer, Matching extension and the genus of a
graph, J. Combin. Theory Ser. B 44 (1986) 329-337.
\bibitem{Plummer86} M.D. Plummer, Matching extension in bipartite graphs, in:
Proceedings of the Seventeenth Southeastern Conference on
Combinatorics, Graph Theory and Computing, in: Congr. Numer., LIV
Utilitas Math., Winnipeg, 1986, pp. 245-258.
\bibitem{Plummer80} M.D. Plummer, On $n$-extendable graphs, Discrete
Math. 31 (1980) 201-210.
\bibitem{M. Porteous95} M. Porteous, Generalizing matching extensions,
M.A. Thesis,  University of Otago, 1995.
\bibitem{M. Porteous96} M. Porteous and R. Aldred, Matching
extensions with prescribed and forbidden edges, Austral. J. Combin.
13 (1996) 163-174.
\bibitem{J.W.T.63} J.W.T. Youngs, Minimal imbeddings and the genus of a
graph, J. Math. Mech. 12 (1963) 303-315.
\bibitem{Zhang08} Heping Zhang, Qiuli Li and Wenwen Liu, Notes on the matching extendability of graphs on surfaces,
submitted.
\end{thebibliography}
\end{document}